\documentclass[12pt,reqno]{amsart}
\usepackage{amsmath, amsthm, amssymb, stmaryrd}

\advance \topmargin by -\headheight
\advance \topmargin by -\headsep
     
\evensidemargin \oddsidemargin
\newtheorem{theorem}{Theorem}[section]
\newtheorem{proposition}[theorem]{Proposition}
\newtheorem{lemma}[theorem]{Lemma}
\newtheorem{corollary}[theorem]{Corollary}

\newtheorem*{question*}{Question}

\theoremstyle{definition}

\theoremstyle{remark}

\numberwithin{equation}{section}

\newcommand{\mmod}[1]{\,\,(\text{\rm mod}\,\,#1)}

\def\bfk{{\mathbf k}}

\def\calP{{\mathcal P}}

\def\mtil{{\widetilde m}}

\def\dbN{{\mathbb N}}  

\def\dbZ{{\mathbb Z}}\def\dbQ{{\mathbb Q}}

\def\grp{{\mathfrak p}}

\def\alp{{\alpha}} 
\def\bet{{\beta}}  
\def\gam{{\gamma}} \def\Gam{{\Gamma}}

\def\zet{{\zeta}}  
 
\def\tet{{\theta}}  
 
\def\kap{{\kappa}}
\def\lam{{\lambda}} \def\lamtil{{\widetilde \lam}}
\def\Lam{{\Lambda}} 
\def\mutil{{\widetilde \mu}}

\def\Mu{{\mathrm M}}

 \def\Ome{{\Omega}}

\def\eps{\varepsilon}

\def\le{\leqslant} \def\ge{\geqslant}

\begin{document}
\title[Smooth values of polynomials]{Smooth values of polynomials}
\author[J. W. Bober, D. Fretwell]{J. W. Bober, D. Fretwell}
\address{JWB and DF: School of Mathematics, University of Bristol, University Walk, Clifton, 
Bristol BS8 1TW, United Kingdom, and the Heilbronn Institute for Mathematical Research, 
Bristol, United Kingdom}
\email{j.bober@bristol.ac.uk, daniel.fretwell@bristol.ac.uk}
\author[G. Martin]{G. Martin}
\address{GM: Department of Mathematics, University of British Columbia, Room 121, 1984 
Mathematics Road, Vancouver, BC, Canada V6T 1Z2}
\email{gerg@math.ubc.ca}
\author[T. D. Wooley]{T. D. Wooley}
\address{TDW: School of Mathematics, University of Bristol, University Walk, Clifton, Bristol 
BS8 1TW, United Kingdom}
\email{matdw@bristol.ac.uk}
\date{}
\begin{abstract} Given $f\in \dbZ[t]$ of positive degree, we investigate the existence of 
auxiliary polynomials $g\in \dbZ[t]$ for which $f(g(t))$ factors as a product of polynomials 
of small relative degree. One consequence of this work shows that for any quadratic 
polynomial $f\in \dbZ[t]$ and any $\eps>0$, there are infinitely many $n\in \dbN$ for which 
the largest prime factor of $f(n)$ is no larger than $n^\eps$.\end{abstract}
\maketitle

\section{Introduction} In this paper we study the smoothness of polynomials. Recall that an 
integer is called $y$-{\it smooth} (or $y$-{\it friable}) when each of its prime divisors is less 
than or equal to $y$. Given a polynomial $f\in \dbZ[t]$ of positive degree, and a 
non-negative number $\tet$, we say that $f$ {\it admits smoothness} $\tet$ when there are 
infinitely many integers $n$ for which the polynomial value $N=|f(n)|$ is 
$N^\tet${\hskip-.7mm}-smooth. Similarly, we say that $f$ {\it admits polysmoothness} 
$\tet$ when there exists a non-constant polynomial $g\in \dbZ[t]$ having the property that 
each irreducible factor of $f(g(t))$ has degree at most $\tet (\deg f)(\deg g)$. In the latter 
circumstances, by inspecting the values $f(g(m))$ for large integers $m$, it is apparent that 
when $f$ admits polysmoothness $\tet$, then it admits smoothness $\eta$ for any 
$\eta>\tet$. Motivated by the widely held conjecture that for each $\eps>0$, every 
$f\in \dbZ[t]$ of positive degree should admit smoothness $\eps$, the latter  considerations 
prompt the following question.

\begin{question*} Given $f\in \dbZ[t]$ of positive degree $d$, is it the case that $f$ admits 
polysmoothness $\eps$ for every $\eps>0$? In other words, for each $\eps>0$, does there 
exist $g\in \dbZ[t]$ of some degree $k=k(\eps)\ge 1$ having the property that each 
irreducible factor of $f(g(t))$ has degree at most $\eps kd$?
\end{question*}

If the answer to this question is in the affirmative, then the aforementioned smoothness 
conjecture on polynomial values would follow at once. Regrettably, with the exception of 
polynomials of special shape, an affirmative answer has been available only in the case 
$d=1$. Our primary goal in this paper is to answer this question in the affirmative in the 
case $d=2$.

\begin{theorem}\label{theorem1.1} Let $f\in \dbZ[t]$ be quadratic. Then for some $c>0$ 
there are polynomials $g\in \dbZ[t]$ of arbitrarily large odd degree $k$ for which $f(g(t))$ 
factors as a product of polynomials of degree at most $ck/\sqrt{\log\log k}$. Thus $f$ 
admits polysmoothness $\eps$ for any $\eps>0$.
\end{theorem}

\begin{corollary}\label{corollary1.2} When $\eps>0$ and $f\in \dbZ[t]$ is quadratic, there 
are infinitely many $n\in \dbN$ for which $f(n)$ is $n^\eps$-smooth. Thus $f$ admits 
smoothness $\eps$.
\end{corollary}

The sharpest conclusion available for quadratic polynomials hitherto is due to Schinzel 
\cite[Theorem 15]{Sch1967}. This work, half a century old, shows that when $f\in \dbZ[t]$ 
is quadratic, then it admits smoothness $\tet$, where
$$\tet=\frac{1}{2}\left( 1-\frac{1}{3}\right)\left(1-\frac{1}{7}\right) \left(1-\frac{1}{47}
\right) \left( 1-\frac{1}{2207}\right) \cdots =0.27950849\ldots.$$
For certain classes of quadratic polynomials one can do better. For example, Schinzel 
\cite[Theorem 14]{Sch1967} shows that if $f(t)=a(rt+s)^2\pm b$, with $a,r,s\in \dbZ$, 
$ar\ne 0$ and $b\in\{1,2,4\}$, then $f$ admits smoothness $\eps$ for any $\eps>0$. 
However, as is implict in the concluding remarks of Schinzel \cite{Sch1967}, polynomials such 
as $4t^2+4t+9=(2t+1)^2+8$ remain inaccessible to these methods. It is for awkward 
polynomials of this type that Corollary \ref{corollary1.2} for the first time confirms the 
longstanding smoothness conjecture.\par

The state of knowledge for polynomials of degree exceeding $2$ is in general far less 
satisfactory. Discussion here requires that we return to the topic of well-factorable polynomial 
compositions. Consider a polynomial $f\in \dbZ[t]$ of degree $d\ge 2$. Then the simplest 
approach generally applicable stems from the trivial identity $f(t)\equiv 0\mmod{f(t)}$, which 
yields the only slightly less trivial relation $f(t+f(t))\equiv 0\mmod{f(t)}$. The latter is of 
course merely another means of expressing the factorisation $f(t+f(t))=f(t)h(t)$, where 
$h\in \dbZ[t]$ is some polynomial of degree $d^2-d$. This instantly shows that $f$ admits 
polysmoothness $1-1/d$, but more can be extracted by iterating this construction. Thus, in 
the next step, one substitutes $t=x+h(x)$ into the polynomial $f(t+f(t))$, and so on. In this 
way one sees that a polynomial $g\in \dbZ[t]$ may be found having the property that 
$f(g(t))$ has as many irreducible factors as desired, and moreover that $f$ admits 
polysmoothness $\tet(d)$, with
$$\tet(d)=1-\frac{1}{d-1}+O\left( \frac{1}{d^3}\right).$$
\par Schinzel \cite[Lemma 10]{Sch1967} offers a more elaborate construction, which we will 
revisit in \S3, showing that for a degree $d$ polynomial $f\in \dbZ[t]$, there exists a degree 
$d-1$ polynomial $g\in \dbZ[t]$ having the property that $f(g(t))$ has a degree $d$ factor. 
This construction may also be iterated. In order to describe the limit of Schinzel's circle of 
ideas, we introduce some notation. When $d\ge 2$, we define the sequence 
$(u_i)_{i=1}^\infty$ by putting $u_1=d-1$, and then setting $u_{i+1}=u_i^2-2$ $(i\ge 1)$. 
We may now define the exponent $\tet(d)$ by taking
$$\tet(d)=\begin{cases} \tfrac{1}{2}P(2d),&\text{when $d=2,3$,}\\
P(d),&\text{when $d>3$,}\end{cases}$$
where
\begin{equation}\label{1.1}
P(d)=\prod_{i=1}^\infty (1-1/u_i).
\end{equation}
Then Schinzel \cite[Theorem 15]{Sch1967} shows that every polynomial $f\in \dbZ[t]$ of 
degree $d$ admits polysmoothness $\tet(d)$. A modest computation reveals that
$$\tet(2)=0.27950849\ldots,\quad \tet(3)=0.38188130\ldots ,\quad \tet(4)=0.55901699\ldots 
,$$
and that for large $d$ one has
$$\tet(d)=1-\frac{1}{d-2}+O\left( \frac{1}{d^3}\right).$$
Although the conclusion of Theorem \ref{theorem1.1} supercedes this result in the case 
$d=2$, further progress for larger degrees remains elusive. This absence of progress for 
larger degrees motivates the exploration of families of polynomials $f$ admitting sharper 
smoothness than attained via Schinzel's construction. In \S3 we reinterpret Schinzel's 
method in terms of field structures associated with the splitting field for $f$ over 
$\dbQ$. Thereby, we obtain some enhancements applicable for special families of 
polynomials summarised in the following result.

\begin{theorem}\label{theorem1.3} Let $f\in \dbZ[t]$ be irreducible, and let $\alp$ be a root 
of $f$ lying in its splitting field. Suppose that for some $\gam\in \dbQ(\alp)$ and 
$g\in \dbZ[t]$ of degree $k\ge 2$, one has $\alp=g(\gam)$. Then $f(g(t))$ is divisible by 
the minimal polynomial of $\gam$ over $\dbQ$, and hence $f$ admits polysmoothness 
$1-1/k$. 
\end{theorem} 

Suppose that $f\in \dbZ[t]$ is irreducible of degree $d$, and that $\alp$ is a root of $f$ lying 
in its splitting field. Then given any $\gam \in \dbQ(\alp)$ with 
$\dbQ(\gam)=\dbQ(\alp)$, since $\alp\in \dbQ(\gam)$, we find that there exists a polynomial 
$g\in \dbQ[t]$ of degree at most $d-1$ with $\alp=g(\gam)$. Thus, save for establishing 
that $\text{deg}(g)>1$ and further that the coefficients of $g$ may be taken to be integers, 
Theorem \ref{theorem1.3} recovers the conclusion of Schinzel. It is apparent, moreover, that 
there is the potential for this polynomial $g$ to have degree significantly smaller than that of 
$f$, and in such circumstances one does better than Schinzel.

\begin{corollary}\label{corollary1.4} Suppose that $f\in \dbZ[t]$ is irreducible, and let $\alp$ 
be a root of $f$ lying in its splitting field. Suppose that $f(t)=g(h(t))-t$, with 
$g,h\in \dbZ[t]$ of degree exceeding $1$. Then $f(g(t))$ is divisible by the minimal 
polynomial of $h(\alp)$ over $\dbQ$, and hence $f$ admits polysmoothness 
$1-1/\text{\rm deg}(g)$.
\end{corollary}

The point here is that, since $f(\alp)=0$, one has $\alp=g(h(\alp))$, and so one can apply 
Theorem \ref{theorem1.3} with $\gam=h(\alp)$. Thus, for example, the polynomial 
$f(t)=t^4+4t^2-t+1$ satisfies the relation
$$f(t)=(t^2+1)^2+2(t^2+1)-t-2=g(h(t))-t,$$
with $g(t)=t^2+2t-2$ and $h(t)=t^2+1$. One may verify that $f$ is irreducible over $\dbQ$. 
Hence, if $\alp$ is a root of $f$ lying in its splitting field, one deduces from the 
corollary that $f(t^2+2t-2)$ is divisible by the minimal polynomial of $h(\alp)$ over $\dbQ$. 
Note that since $\alp=g(h(\alp))$, it is not possible that $h(\alp)$ lies in a proper subfield of 
$\dbQ(\alp)$, and hence its minimal polynomial has degree $4$. In this way, one finds that 
$f(t^2+2t-2)=m_1(t)m_2(t)$, for polynomials $m_i\in \dbZ[t]$ each having degree $4$. 
Indeed, one has
$$f(t^2+2t-2)=(t^4+4t^3-9t+5)(t^4+4t^3-7t+7).$$
Thus $f$ admits polysmoothness $\tfrac{1}{2}$, which is already sharper 
than the conclusion of Schinzel, which shows $f$ to admit polysmoothness at most 
$0.559\ldots $. From here, one can continue by iterating Schinzel's construction on 
$m_1m_2$, thereby showing that $f$ admits polysmoothness $\tfrac{1}{2}P(8)$, where 
$P(8)$ is defined as in (\ref{1.1}). In this way, one may verify that $f$ admits 
polysmoothness $0.41926274\ldots$.\par

The method underlying the proof of Theorem \ref{theorem1.1} may be seen as a hybrid of 
the cyclotomic construction of \S2 with the field theoretic approach described in \S3. This we 
describe in \S4. Another class of polynomials is susceptible to a decomposition in some ways 
reminiscent of the Aurifeuillian factorisations discussed by Granville and Pleasants in 
\cite{GP2006}. We illustrate our ideas in \S5 with the simplest classes of trinomials.

\begin{theorem}\label{theorem1.5} Suppose that $k$ is a natural number with $k\ge 2$.
\begin{enumerate}
\item[(i)] Let $f(t)=t^k+at^{k-1}-b$, with $a,b\in \dbZ$ and $b\ne 0$. Then $f$ admits 
polysmoothness $\phi(k-1)/(k-1)$.
\item[(ii)] Let $f(t)=at^k-t+b$, with $a,b\in \dbZ$ and $ab\ne 0$. Then $f$ admits 
polysmoothness $\phi(k)/k$.
\end{enumerate}
\end{theorem}

To illustrate the potential effectiveness of this theorem, consider the polynomial 
$f_k(t)=t^k-t-1$. It was shown by Selmer \cite{Sel1956} that $f_k$ is irreducible over 
$\dbQ$ for each $k\ge 2$. Meanwhile, Theorem \ref{theorem1.5} shows that $f$ admits 
polysmoothness $\phi(k)/k$, and this can be made arbitrarily close to $0$ by taking a 
sequence of exponents $k$ equal to the product of the first $n$ prime numbers, and letting 
$n\rightarrow \infty$. In this special case, the method of proof is simple to describe. We 
take a polynomial $g$ of large degree, and put $t=g(x)^k-1$. Thus we find that
$$f_k(g(x)^k-1)=(g(x)^k-1)^k-g(x)^k,$$
and so, as a difference of two $k$-th powers, we may employ a decomposition via 
cyclotomic polynomials to factorise $f_k(g(x)^k-1)$.\par

In our discussion of smooth values of polynomials, we have emphasised the application of 
polysmoothness to establish smoothness. An alternative approach to the problem of showing 
that polynomials take smooth values at integral arguments is via sieve theory. Typical of the 
kind of result that may be established is a conclusion of Dartyge, Tenenbaum, and the third 
author \cite{DMT2001}. Let $f\in \dbZ[t]$ be a polynomial of degree $d\ge 1$, and let 
$\eps>0$. Then, in particular, these authors show that for a positive proportion of integers 
$n$, the integer $N=|f(n)|$ is $N^{1-1/d+\eps}$-smooth. Thus $f$ admits smoothness 
$1-1/d+\eps$. Although weaker than the conclusions described above, this smoothness result 
has the merit that it applies for a positive proportion of the values represented by $f$. In the 
same vein, subject to the truth of a certain uniform quantitative form of the 
Schinzel--Sierpi\'nski Hypothesis, the third author \cite{Mar2002} has obtained an asympotic 
formula for the number of integers $n$ with $1\le n\le x$ for which $N=|f(n)|$ is 
$N^{1-1/(d-1)+\eps}$-smooth when $f$ is irreducible.
\vskip.1cm

\noindent{\bf Acknowledgements:} The first, third and fourth authors are grateful to the 
Heilbronn Institute for Mathematical Research for funding a visit of the third author to the 
University of Bristol in 2014 that laid the foundations for the work reported here. The third 
author's work is partially supported by a National Sciences and Engineering Research 
Council of Canada Discovery Grant. The fourth author's work is supported by a European 
Research Council Advanced Grant under the European Union's Horizon 2020 research and 
innovation programme via grant agreement No.~695223. The authors are grateful to 
Michael Stoll for providing the sextic example concluding \S6. 

\section{A cyclotomic construction} The polysmoothness question described in the 
introduction can be answered for polynomials $f\in \dbZ[t]$ equal to any product of 
binomials. Consider then integers $a_j$, $b_j$, $k_j$ with $k_j\ge 1$ $(1\le j\le l)$, and 
the polynomial
\begin{equation}\label{2.1}
f(t)=\prod_{j=1}^l(a_jt^{k_j}-b_j).
\end{equation}
The argument employed by Balog and the fourth author in their proof of \cite[Lemma 2.2]
{BW1998} is easily modified to confirm that $f$ admits polysmoothness $\eps$ for any 
$\eps>0$, as we now show.

\begin{theorem}\label{theorem2.1} Let $f\in \dbZ[t]$ be a polynomial of the shape 
(\ref{2.1}) with $a_1\cdots a_l\ne 0$. Then for some $c=c(\bfk)>0$, there are polynomials 
$g\in \dbZ[t]$ of arbitrarily large degree $d$ for which $f(g(t))$ factors as a product of 
polynomials of degree at most $cd/(\log \log d)^{1/l}$. Thus $f$ admits polysmoothness 
$\eps$ for any $\eps>0$.
\end{theorem}

\begin{proof} Write $k=k_1\cdots k_l$, and let $y$ be a natural number sufficiently large in 
terms of $k$. Then it follows from \cite[Lemma 2.1]{BW1998} that the prime numbers not 
exceeding $y$ and coprime to $k$ can be partitioned into $l$ sets $\calP_1,\ldots ,\calP_l$ 
with the property that for each $i$, one has
\begin{equation}\label{2.2}
\prod_{p\in \calP_i}(1-1/p)<2\biggl( \frac{k}{\phi(k)\log y}\biggr)^{1/l}\quad \text{and}
\quad \prod_{p\in \calP_i}p<y^2e^{5y/(4l)}.
\end{equation}
Put
$$\gam_i=\prod_{p\in \calP_i}p\quad \text{and}\quad \Gam_i=\prod_{\substack{1\le j\le l\\
j\ne i}}\gam_j,$$
and write $\Gam=\gam_1\cdots \gam_l$. It follows from standard prime number estimates 
that $k^{-1}e^{3y/4}<\Gam<e^{5y/4}$. Since $(k_j\Gam_j,\gam_j)=1$ $(1\le j\le l)$, we 
find that for each index $j$ there exist integers $\lamtil_j,\mutil_j$ with 
$1\le \lamtil_j,\mutil_j<\gam_j$, and satisfying
$$k_j\Gam_j\lamtil_j\equiv -1\mmod{\gam_j}\quad \text{and}\quad 
k_j\Gam_j\mutil_j\equiv 1\mmod{\gam_j}.$$
We put $\lam_j=\Gam_j\lamtil_j$ and $\mu_j=\Gam_j\mutil_j$ $(1\le j\le l)$. Also, when 
$1\le i,j\le l$ and $i\ne j$, we define the integers $\Lam_{ij}$ and $\Mu_{ij}$ via the relations
$$\Lam_{ij}=k_j\lam_i/\gam_j\quad \text{and}\quad \Mu_{ij}=k_j\mu_i/\gam_j,$$
and
$$\Lam_{jj}=(k_j\lam_j+1)/\gam_j\quad \text{and}\quad \Mu_{jj}=(k_j\mu_j-1)/\gam_j.$$

\par We are now equipped to define the auxiliary polynomial
$$g(t)=t^\Gam \prod_{j=1}^la_j^{\lam_j}b_j^{\mu_j}.$$
This polynomial has degree $\Gam$ satisfying $k^{-1}e^{3y/4}<\Gam<e^{5y/4}$, whence
\begin{equation}\label{2.3}
y\asymp \log \Gam.
\end{equation}
Moreover, when $1\le j\le l$, one has $a_jg(t)^{k_j}-b_j=b_j(z_j^{\gam_j}-1)$, where
$$z_j=t^{k_j\Gam_j}\prod_{i=1}^la_i^{\Lam_{ij}}b_i^{\Mu_{ij}}.$$
But $z_j^{\gam_j}-1=\prod_{e|\gam_j}\Phi_e(z_j)$, where $\Phi_e$ denotes the $e$-th 
cyclotomic polynomial. It therefore follows that
$$f(g(t))=\prod_{j=1}^lb_j(z_j^{\gam_j}-1)$$
factors as a product of polynomials of degree at most
$$\max_{1\le j\le l}\max_{e|\gam_j}\phi(e)\text{\rm deg}(z_j)=\max_{1\le j\le l}k_j\Gam_j
\phi(\gam_j)=\Gam \max_{1\le j\le l}\frac{k_j\phi(\gam_j)}{\gam_j}.$$
But in view of (\ref{2.2}) and (\ref{2.3}), one has
$$\frac{k_j\phi(\gam_j)}{\gam_j}=k_j\prod_{p\in \calP_j}(1-1/p)<2k_j
\biggl( \frac{k}{\phi(k)\log y}\biggr)^{1/l}\ll (\log \log \Gam)^{-1/l}.$$
Then we are forced to conclude that there is a number $c=c(\bfk)>0$ for which $f(g(t))$ 
factors as a product of polynomials of degree at most $c\Gam/(\log \log \Gam)^{1/l}$, 
where $\Gam=\text{\rm deg}(g)$. This completes the proof of the theorem.
\end{proof}

The special case of Theorem \ref{theorem2.1} corresponding to the polynomial
$$f(t)=(a_1t-b_1)(a_2t-b_2),$$
with $l=2$, $k_1=k_2=1$, confirms the conclusion of Theorem \ref{theorem1.1} in the 
special case of quadratic polynomials that factor as a product of two linear factors. We may 
consequently restrict attention in our proof of Theorem \ref{theorem1.1} in \S4 to irreducible 
quadratic polynomials.

\section{Field theoretic constructions} We begin in this section by describing the proof of 
Theorem \ref{theorem1.3}. This permits an abstract explanation of the construction of 
Schinzel described in the introduction, though for the sake of simplicity we restrict ourselves 
in such matters to monic polynomials.

\begin{proof}[The proof of Theorem \ref{theorem1.3}] Let $f\in \dbZ[t]$ be irreducible of 
degree $d\ge 2$, and let $\alp$ be a root of $f$ lying in its splitting field. Suppose 
that $\gam\in \dbQ(\alp)$, and that for some $g\in \dbZ[t]$ of degree $k\ge 2$, one has 
$\alp=g(\gam)$. Since $f(g(\gam))=f(\alp)=0$, it follows that the minimal polynomial $m$ 
of $\gam$ over $\dbQ$ divides $f(g)$. By Gauss' Lemma, we infer that $f(g)$ is divisible by 
an integral multiple $\mtil$ of $m$ lying in $\dbZ[t]$. One has
$$\dbQ(\gam)\subseteq \dbQ(\alp)=\dbQ(g(\gam))\subseteq \dbQ(\gam),$$
and hence
$$\text{deg}(\mtil)=[\dbQ(\gam):\dbQ]=[\dbQ(\alp):\dbQ]=\text{deg}(f)=d.$$
Then for some polynomial $h\in \dbZ[t]$, one has $f(g)=\mtil h$, with $\text{deg}(\mtil)=d$, 
$\text{deg}(f(g))=kd$ and $\text{deg}(h)=kd-d$. We thus conclude that every polynomial 
factor of $f(g)$ has degree at most $d(k-1)$, whence $f$ admits polysmoothness $1-1/k$. 
This completes the proof of Theorem \ref{theorem1.3}.
\end{proof}

The factorisation of $f(g(x))$ is in general closely related to the factorisation of $g(x)-\alp$, 
as various authors have noticed. The conclusion of Theorem \ref{theorem1.3} is closely 
related to the following proposition, which Schinzel \cite[Theorem 22]{Sch2000} attributes to 
Capelli. (This proposition also appears, in a slightly infelicitous form, as 
\cite[Lemma 1]{GP2006}.)

\begin{proposition}\label{lemma3.1} Let $f\in \dbQ[t]$ be monic and irreducible, let $\alp$ 
be any root of $f$ in its splitting field, and put $K=\dbQ(\alp)$. Then, for any 
$g\in \dbQ[t]$, if the factorisation of $g(t)-\alp$ as a product of irreducibles over $K[t]$ is
$$g(t)-\alp=a_1(t;\alp)^{r_1}\cdots a_k(t;\alp)^{r_k},$$
then the factorisation of $f(g(t))$ as a product of irreducibles over $\dbQ[t]$ is
$$f(g(t))=A_1(t)^{r_1}\cdots A_k(t)^{r_k},$$
with
$$A_j(t)=\prod_{f(\bet)=0}a_j(t;\bet)\quad (1\le j\le k).$$
\end{proposition}

Theorem \ref{theorem1.3} follows from Proposition \ref{lemma3.1} as the special case in 
which one of the irreducible factors $a_i(t;\alp)$ is linear. It is apparent that, in the setting of 
Proposition \ref{lemma3.1}, the wider generality that it has the potential to offer may be 
exploited to improve polysmoothness bounds for $f$ whenever one has corresponding 
polysmoothness bounds for polynomials $g(t)-\alp$ over $\dbQ(\alp)$.\par

The idea underlying the construction of Schinzel described in \cite[Lemma 10]{Sch1967} can 
be interpreted in the guise of Theorem \ref{theorem1.3} as follows. We restrict attention to 
monic irreducible polynomials
$$f(t)=t^d+a_{d-1}t^{d-1}+\ldots +a_1t+a_0.$$
Let $\alp$ be a root of $f$ lying in its splitting field, and put $\bet=1/\alp$. Then 
since $f(\alp)=0$, it follows that
$$\alp=-(a_{d-1}+a_{d-2}\bet+\ldots +a_0\bet^{d-1}).$$
We take $g(t)=-(a_0t^{d-1}+\ldots +a_{d-1})$. Then $\alp=g(\bet)$ with $g\in \dbZ[t]$ a 
polynomial of degree $d-1$. Theorem \ref{theorem1.3} consequently delivers the 
conclusion that $f$ admits polysmoothness $1-1/(d-1)$.\par

Two comments are in order here. First, the restriction that $f$ be irreducible is easily 
negotiated away with a little careful thought. Second, the condition that $f$ be monic in the 
above argument can be surmounted with the application of carefully chosen shifts, as 
Schinzel demonstrates. This is a little delicate, and we have chosen to avoid such technical 
issues with the hope that the underlying ideas may be more clearly visible from our simplified 
discussion.\par

Incidentally, the strategy employed in the above argument is relevant to the question raised 
by Granville and Pleasants following \cite[Corollary 1]{GP2006}:

\begin{question*}[Granville and Pleasants] Suppose that $f(t)\in \dbQ[t]$ is irreducible. Can 
one find infinitely many $g(y)\in \dbQ[y]$ with $\text{\rm deg}(g)<\text{\rm deg}(f)$ for 
which $f(g(y))$ is reducible in $\dbQ[y]$, where the $g(y)$ are distinct under 
transformations replacing $y$ by a polynomial in $y$?
\end{question*}

When $f\in \dbQ[t]$ is irreducible of degree $2$, and $g\in \dbQ[y]$ has degree smaller than 
that of $f$, it is apparent that $g$ is linear, and hence the answer to this question is 
negative. Suppose then that $f$ has degree $d\ge 3$, and let $\alp$ be a root of $f$ in its 
splitting field. Thus $[\dbQ(\alp):\dbQ]=d$. We take $\bet$ to be any element of 
$\dbQ(\alp)$ not lying in $\text{\rm span}_\dbQ\{1,\alp\}$ with $[\dbQ(\bet):\dbQ]=d$. 
There are infinitely many such elements $\bet$. Then since 
$\dbQ(\bet)\subseteq \dbQ(\alp)$ and $[\dbQ(\bet):\dbQ]=[\dbQ(\alp):\dbQ]$, one has  
$\dbQ(\bet)=\dbQ(\alp)$. But $\alp\in \dbQ(\bet)$, so there exists a polynomial 
$g\in \dbQ[y]$ of degree at most $d-1$ with the property that $\alp=g(\bet)$. Furthermore, 
since $\bet\not\in \text{\rm span}_\dbQ\{1,\alp\}$, one sees that $\text{\rm deg}(g)>1$. 
We have $f(g(\bet))=f(\alp)=0$, so that $f(g(y))$ is divisible by the minimal polynomial of 
$\bet$ over $\dbQ$. Since the latter has degree $[\dbQ(\bet):\dbQ]=d$, it follows that 
$f(g(y))$ is reducible, and yet $\text{\rm deg}(g)\le d-1<\text{\rm deg}(f)$. What is unclear 
is whether or not the polynomials $g$ generated by this process are distinct under 
polynomial transformations, although it seems unlikely that all of these polynomials could be 
generated by a finite set by such substitutions. In the cubic case, however, we are able to 
resolve this issue.

\begin{theorem}\label{theorem3.2}
Suppose that $f(t)\in \dbQ[t]$ is irreducible of degree $3$. Then there are infinitely many 
quadratic polynomials $g(y)\in \dbQ[y]$ for which $f(g(y))$ is reducible in $\dbQ[y]$, where 
the $g(y)$ are distinct under transformations replacing $y$ by a polynomial in $y$.
\end{theorem}

\begin{proof} We follow the construction described above, employing the same notation, and 
initially seek a more detailed description of the factorisations of the compositions in question. 
Thus, for a quadratic polynomial $g\in \dbQ[y]$, one finds that $f(g(y))$ is divisible by the 
minimal polynomial $m_\bet$ of $\bet$ over $\dbQ$, and $\text{\rm deg}(m_\bet)=3$. Thus 
$f(g(y))=m_\bet(y)l(y)$, for some polynomial $l\in \dbQ[y]$ of degree $3$. Note that 
$\bet\in \dbQ(\alp)$ is one root of the quadratic polynomial $g(y)-\alp$. The second root 
$\gam$ must also lie in $\dbQ(\alp)$. We observe that $\gam$ cannot be a root of 
$m_\bet$, for then one would have that $g(y)-\alp$ divides $m_\bet(y)$. The quotient 
$q(y)$ here is linear, and cannot lie in $\dbQ[y]$, since $m_\bet(y)$ is irreducible over 
$\dbQ[y]$. But the leading coefficient of $q(y)$ is rational, so the coefficient of $y^2$ in 
$m_\bet(y)$ cannot be rational, leading to a contradiction. Thus, indeed, we have 
$m_\bet(\gam)\ne 0$, confirming our earlier assertion.\par

The polynomial $l$ cannot have linear or quadratic factors over $\dbQ[y]$, for any root 
$\tet$ of such a factor would supply a root $g(\tet)$ of $f$ lying in a field extension of 
$\dbQ$ of degree $1$ or $2$, contradicting the irreducibility of $f$. Then $l(y)$ is a scalar 
multiple of a cubic polynomial irreducible over $\dbQ[y]$. But $\gam$ is a root of this 
polynomial, so that $l$ is a scalar multiple of its minimal polynomial $m_\gam$, and we have 
$f(g)=\kap m_\bet m_\gam$, for some non-zero rational number $\kap$. Moreover, since 
$g\in \dbQ[y]$ and both $\bet$ and $\gam$ are roots of the quadratic polynomial 
$g(y)-\alp$, then an examination of the coefficient of $y$ in the latter polynomial reveals that 
$\bet+\gam\in \dbQ$. There is therefore a rational number $r$ for which $\gam=r-\bet$, 
and we have $f(g(y))=\kap m_\bet(y)m_{r-\bet}(y)$.\par

We next attend to the matter of confirming that infinitely many of these polynomials $g$ are 
distinct under transformations replacing $y$ by a polynomial in $y$. It is apparent that the 
only possibility for such a transformation is a linear one taking $y$ to $ay+b$ for some 
rational numbers $a$ and $b$ with $a\ne 0$. Motivated by this observation, when 
$F,G\in \dbQ[y]$, we write $F\sim G$ when there exist $a,b\in \dbQ$ with $a\ne 0$ for which 
$F(y)=G(ay+b)$. It is readily confirmed that this relation defines an equivalence relation on 
elements of $\dbQ[y]$. Returning now to the discussion of the previous paragraph, one may 
check that
$$f\left(g\left( \frac{y-b}{a}\right)\right) =\kap m_{a\bet+b}(y)m_{a(r-\bet)+b}(y).$$
Hence, whenever $G\sim g$, then $f(G(y))=\kap h_1(y)h_2(y)$ for some monic polynomials 
$h_1,h_2\in \dbQ[y]$ with $h_i\sim m_\bet$ $(i=1,2)$.\par

Suppose that $\bet=A\alp^2+B\alp+C$ and $\bet'=A'\alp^2+B'\alp+C'$, with 
$A,B,C\in \dbQ$ satisfying $A\ne 0$, and likewise for the decorated analogues of these 
coefficients. Consider the composition factorisations
$$f(g(y))=\kap m_\bet(y)m_{r-\bet}(y)\quad \text{and}\quad f(g'(y))=\kap' m_{\bet'}(y)
m_{r'-\bet'}(y)$$
induced from these elements by the process described above. If $g\sim g'$, then the 
conclusion of the previous paragraph shows that one must have $m_{\bet'}\sim m_\bet$. 
It is possible that $\bet$ is the only root of $m_\bet$ lying in $\dbQ(\alp)$, in which case 
we see that for some $a,b\in \dbQ$ with $a\ne 0$, one must have $\bet'=a\bet+b$. Thus, 
since $[\dbQ(\alp):\dbQ]=3$, it follows that $A'=aA$ and $B'=aB$, whence $B'/A'=B/A$. 
In such circumstances, it follows that the equivalence classes for $g$ are classified by 
distinct ratios $B/A$, of which there are infinitely many, and the conclusion of the theorem 
follows. It is possible, meanwhile, that $m_\bet$ splits over $\dbQ(\alp)[y]$. One then 
has $m_\bet(y)=(y-\bet_1)(y-\bet_2)(y-\bet_3)$, with $\bet_i=A_i\alp^2+B_i\alp+C_i$, 
for suitable rational coefficients $A_i,B_i,C_i$. In such circumstances, a similar argument to 
that just employed reveals that for suitable rational numbers $a$ and $b$ with $a\ne 0$, and 
for some $i\in \{1,2,3\}$, one has $\bet'=a \bet_i+b$. In particular, one sees that 
$A_i\ne 0$ and $B'/A'=B_i/A_i$. Consequently, were there to be at most $N$ distinct 
equivalence classes for the polynomials $g$ generated by choices for 
$\bet=A\alp^2+B\alp+C$, then the number of possible ratios $B_i/A_i$ occurring amongst 
the associated roots $\bet_i=A_i\alp^2+B_i\alp+C_i$ would be at most $3^N$. Since there 
are infinitely many such ratios available to us, we derive a contradiction to the hypothesis 
that the number of equivalence classes is finite. Once again, therefore, we obtain the 
conclusion of the theorem.
\end{proof} 

We are confident that a somewhat more elaborate argument would establish the quartic 
analogue of this theorem. Degrees exceeding $4$, on the other hand, would appear to be 
substantially more challenging. We finish this section by establishing Corollary 
\ref{corollary1.4}.

\begin{proof}[The proof of Corollary \ref{corollary1.4}] Assume the hypotheses of the 
statement of Corollary \ref{corollary1.4}, so that $f(t)=g(h(t))-t$. Let $\alp$ be a root of 
$f$ lying in its splitting field. Then $\alp=g(h(\alp))$, so one can apply Theorem 
\ref{theorem1.3} with $\gam=h(\alp)$. All that remains is to observe that the minimal 
polynomial of $\gam$ over $\dbQ$ must have degree $\text{deg}(f)=[\dbQ(\alp):\dbQ]$, 
since $\dbQ(\alp)=\dbQ(g(\gam))\subseteq \dbQ(\gam)$.
\end{proof}

\section{Smoothness of quadratic polynomials} Our goal in this section is the proof of 
Theorem \ref{theorem1.1} in the situation that $f\in \dbZ[t]$ is quadratic and irreducible. As 
we have commented already at the end of \S2, this special case is all that we must now 
address in order to complete the proof of Theorem \ref{theorem1.1}. Our argument can be 
construed as a hybrid of the methods discussed in \S\S2 and 3. We begin with an auxiliary 
lemma, the utility of which will become apparent in due course.

\begin{lemma}\label{lemma4.1} Let $f(t)=at^2+bt+c\in \dbZ[t]$ be irreducible with 
$a\ne 0$. Denote by $\alp$ a root of $f$ in its splitting field. Then for any $k\in \dbN$, there 
exist integers $m$, $n$, $A$ and $B$ with $A\ne 0$ and $(A,B)=1$ such that 
$(ma\alp+n)^k=A\alp+B$.
\end{lemma}

\begin{proof} There exists some rational prime $p$ not dividing $a$ which splits in 
$K=\dbQ(\alp)$, so $(p)=\grp_1\grp_2$ with $\grp_1$ and $\grp_2$ contained in the order 
$\dbZ[a\alp]$. Denoting the class number of $K$ by $h(K)$, one has that $\grp_1^{h(K)}$ is 
principal and hence generated by $ma\alp +n$ for some $m,n\in \dbZ$ with $m\ne 0$. Since 
$ma\alp+n$ is an algebraic integer of $K$, it follows that for any $k\in \dbN$, one has 
$(ma\alp+n)^k=A\alp+B$ for some $A,B\in \dbZ$ with $A\ne 0$ and $a|A$. It remains now 
only to confirm that $(A,B)=1$. But since $ma\alp+n$ generates $\grp_1^{h(K)}$ and
$$\text{\rm Norm}_{K/\dbQ}((ma\alp+n)^k)=\text{\rm Norm}_{K/\dbQ}(A\alp+B)=
a^{-1}(aB^2-bAB+cA^2),$$
we find that $(A/a)Ac-(A/a)Bb+B^2$ is a power of $p$. Any prime which divides both $A$ 
and $B$ must divide this norm, and thus must be equal to $p$. However, one cannot have 
both $p|A$ and $p|B$, for then the ideal $(A\alp+B)=(ma\alp+n)^k$ would be divisible by 
the ideal $(p)=\grp_1\grp_2$, contradicting our assumption that $\grp_1^{h(K)}$ is 
generated by $ma\alp+n$. Thus we conclude that $(A,B)=1$.
\end{proof}

\begin{proof}[The proof of Theorem \ref{theorem1.1}] Let $f(t)=at^2+bt+c\in \dbZ[t]$ be 
irreducible, and let $\alp$ and $\alp'$ be the roots of $f$ in its splitting field. We 
have in mind the application of Lemma \ref{lemma4.1} to seek a relation of the shape 
$\alp=(\bet^k-B)/A$ that we hope to apply in a manner not dissimilar to Theorem 
\ref{theorem1.3}. First we describe the powers $k$ in play. We take $X$ to be large, and 
choose $k$ to be the product of all the primes less than $X$ not dividing $2a\phi(a)$. Since 
we are omitting only a finite set of primes, it follows that
\begin{equation}\label{4.1}
\prod_{p|k}(1-1/p)\asymp 1/\log X\asymp 1/\log \log k.
\end{equation}

\par By applying Lemma \ref{lemma4.1}, one finds that there exist integers $m$, $n$, 
$A$, $B$ with $A\ne 0$ and $(A,B)=1$ for which $(ma\alp+n)^k=A\alp+B$. We put 
$\bet=A\alp+B$, and note that $f((\bet-B)/A)=f(\alp)=0$. Denote by $\Ome_d$ the set of 
primitive $d$-th roots of unity. Put $G(t)=(t^k-B)/A$, and let $\zet\in \Ome_d$ for 
some $d|k$. Then we see that
$$f(G((ma\alp+n)\zet ))=f(((ma\alp+n)^k-B)/A)=f(\alp)=0$$
and
$$f(G((ma\alp'+n)\zet ))=f(((ma\alp'+n)^k-B)/A)=f(\alp')=0.$$
Note here that when $\zet$ and $\zet'$ are distinct $k$-th roots of unity, then 
$$(ma\alp+n)\zet\ne (ma\alp+n)\zet'\quad \text{and}\quad (ma\alp+n)\zet \ne 
(ma\alp'+n)\zet'.$$
The first relation is self-evident, whilst the second follows by taking $k$-th powers and 
observing that $A\alp+B\ne A\alp'+B$. It therefore follows that all of the roots of $f(G(t))$ 
are accounted for by $(ma\alp+n)\zet$ and $(ma\alp'+n)\zet$ with $\zet\in \Ome_d$ for 
some $d|k$. Thus, one may write $f(G(t))=C\prod_{d|k}h_d(t)$ for a suitable rational 
number $C$, where
$$h_d(t)=\prod_{\zet\in \Ome_d}\left(t-(ma\alp+n)\zet \right) \left(t-(ma\alp'+n)\zet 
\right).$$
Note here that, by considering conjugation in the field extension $\dbQ(\alp,\zet)/\dbQ$, 
for $\zet\in \Ome_d$, it is apparent that $h_d\in \dbQ[t]$ whenever $d|k$. Moreover, the 
polynomial $h_d$ has degree $2\phi(d)$.\par

The possibility remains of an obstruction to selecting a polynomial $g$ having 
{\it integral} coefficients for which $f(g)$ is well-factorable. In order to address this 
complication, we consider the polynomial $g(t)=G(At+z)$, and seek to select $z$ in such a 
manner that $g\in \dbZ[t]$. Put $K=\dbQ(\alp)$ and consider the norm of the algebraic 
integer $A\alp+B$, namely
$$\text{\rm Norm}_{K/\dbQ}(A\alp+B)=a^{-1}(aB^2-bAB+cA^2)=\left( 
\text{Norm}_{K/\dbQ}(ma\alp+n)\right)^k.$$
By construction, we have $a|A$, and thus we see that $B^2$ is a $k$-th power modulo 
$A/a$. Since $k$ is odd, this observation implies that $B$ is also a $k$-th power modulo 
$d$ for every divisor $d$ of $A/a$. Let $a^\prime$ be the divisor of $A$ given by 
$a'=\lim_{N\rightarrow \infty}(A,a^N)$. Then, in particular, we find that $B$ is a $k$-th 
power modulo $A/a'$. But $k$ is coprime to both $a$ and $\phi(a)$, and hence to the 
order of $(\dbZ/a'\dbZ)^\times$, and thus all integers coprime to $a'$ are necessarily 
$k$-th powers modulo $a'$. We may therefore conclude that $B$ is a $k$-th power modulo 
$A/a'$ and modulo $a'$. Since $A/a'$ and $a'$ are coprime, we discern that $B$ is a $k$-th 
power modulo $A$, say $B\equiv z^k\mmod{A}$.\par

We may now put
$$g(t)=G(At+z)=((At+z)^k-B)/A\in \dbZ[t],$$
and we deduce that $f(g(t))=C\prod_{d|k}h_d(At+z)$. Thus, on recalling (\ref{4.1}), we 
infer that $f(g(t))$ factors as a product of polynomials of degree at most
$$\max_{d|k}\text{\rm deg}(h_d)=\max_{d|k}2\phi(d)=2k\prod_{p|k}(1-1/p)\asymp 
2k/\log \log k.$$
By Gauss' Lemma, moreover, there is no loss in supposing that these polynomial factors lie in 
$\dbZ[t]$. In particular, the polynomial $f$ exhibits polysmoothness $\eps$ for any $\eps>0$. 
By construction, moreover, the polynomial $g$ has odd degree $k$, and so the proof of 
Theorem \ref{theorem1.1} is complete.
\end{proof}

Unfortunately, the construction applied here in the proof of Theorem \ref{theorem1.1} is 
less successful for higher degree polynomials. When $f=at^3+bt^2+ct+d\in \dbZ[t]$ is cubic, 
for example, and $\alp$ is a root of $f$ in its splitting field, then one cannot expect 
that there is an integer $k>1$ for which
$$(ma\alp+n)^k=A\alp+B,$$
for appropriate integers $m$, $n$, $A$ and $B$ with $A\ne 0$. Instead, one can find integers 
$A$, $B$ and $C$ for which
$$(ma\alp+n)^k=A\alp^2+B\alp+C.$$
A plausible plan is then to obtain a relation of the type
$$\lam (ma\alp+n)^{2k}+\mu (ma\alp+n)^k=A\alp+B,$$
for suitable integers $A$, $B$, $\lam$, $\mu$. At best, such an approach would deliver a 
polynomial $g$ of the shape
$$g(t)=(\lam (At+z)^{2k}+\mu (At+z)^k-B)/A\in \dbZ[t]$$
having the property that $f(g(t))$ factors as a product of the shape 
$$Ch_0(At+z)\prod_{d|k}h_d(At+z),$$
wherein $h_0$ has degree $3k$. A priori, this might ensure polysmoothness $\tfrac{1}{2}$ 
at best, and so is not inherently stronger than the approach of Schinzel.

\section{Relatives of Aurifeuillian factorisations} We next describe the proof of Theorem 
\ref{theorem1.5}. This will not detain us for long. Suppose in the first instance that 
$f(t)=t^k+at^{k-1}-b$, with $a,b\in \dbZ$ and $b\ne 0$. We rearrange $f$ in order to 
engineer a cyclotomic construction, writing $f(t)=(t+a)t^{k-1}-b$. Thus, if we set 
$g(t)=b^kt^{k-1}-a$, we find that
$$f(g(t))=b\left( \left(btg(t)\right)^{k-1}-1\right)=b\prod_{d|(k-1)}\Phi_d(btg(t)).$$
The polynomial $f(g)$ has degree $K=k(k-1)$, whilst each irreducible factor of $f(g)$ has 
degree at most
$$\max_{d|(k-1)}\phi(d)k\le K\phi(k-1)/(k-1).$$
Then $f$ admits polysmoothness $\phi(k-1)/(k-1)$. This establishes part (i) of Theorem 
\ref{theorem1.5}.\par

Suppose next that $f(t)=at^k-t+b$, with $a,b\in \dbZ$ and $ab\ne 0$. With the same plan in 
mind as above, we set $g(t)=a^{k+1}t^k+b$, and arrive at the relation
$$f(g(t))=a\left( g(t)^k-(at)^k\right)=a\prod_{d|k}(at)^{\phi(d)}\Phi_d(g(t)/(at)).$$
The term in the product here indexed by $d$ is a polynomial of degree $k\phi(d)$. Thus, 
the polynomial $f(g)$ has degree $K=k^2$, whilst each irreducible factor of $f(g)$ has 
degree at most $\max_{d|k}\phi(d)k\le K\phi(k)/k$. Then $f$ admits polysmoothness 
$\phi(k)/k$. This establishes part (ii) of Theorem \ref{theorem1.5}, and completes the proof 
of the theorem.\vskip.2cm

We remark that Harrington \cite[Theorem 1]{Har2012} has investigated the irreducibility of 
polynomials $f(t)$ of the shape $t^n\pm ct^{n-1}\pm d$ over $\dbZ[t]$. Thus, such 
polynomials are irreducible when $n,c,d\in \dbN$ satisfy 
$$n\ge 3,\quad d\ne c,\quad d\le 2(c-1),\quad (n,c)\ne (3,3)\quad \text{and}\quad 
f(\pm 1)\ne 0.$$
Moreover, Ljunggren \cite[Theorem 3]{Lju1960} has shown that all of the polynomials
$$t^{3n}\pm t\pm 1,\quad t^{3n+1}\pm t\pm 1,\quad t^{6n+5}-t\pm 1\quad \text{and}
\quad t^{6n+2}\pm t-1$$
are irreducible for all natural numbers $n$.

\section{Polynomials resisting polysmoothness}
We finish with an account of some examples demonstrating limitations to the most ambitious 
results one might imagine concerning polysmoothness. We concentrate on irreducible 
polynomials $f_d\in \dbZ[t]$ of degree $d$. In view of the conclusion of Theorem 
\ref{theorem1.1}, it makes sense to restrict attention to degrees $d$ exceeding $2$. One 
might optimistically hope that for each such polynomial, there should exist a quadratic 
polynomial $g\in \dbZ[t]$ having the property that $f_d(g(t))=h_1(t)h_2(t)$, for some 
polynomials $h_i\in \dbZ[t]$ irreducible of degree $d$. Note here that $f_d(g(t))$ cannot be 
divisible by a polynomial $h\in \dbZ[t]$ of degree smaller than $d$, for then a root $\bet$ 
of this polynomial in its splitting field would supply a root $g(\bet)$ of $f_d$ with 
$[\dbQ(g(\bet)):\dbQ]<d$, contradicting the irreducibility of $f_d$. Thus, if the polynomial 
$f_d(g(t))$ is reducible, then necessarily it factors in precisely the shape $h_1(t)h_2(t)$ 
asserted.\par

As we have already discussed in the introduction, the construction of Schinzel 
\cite[Lemma 10]{Sch1967} shows that in the cubic case $d=3$, quadratic polynomials 
$g\in \dbZ[x]$ can be found for which $f_3(g(x))=h_1(x)h_2(x)$, with $h_1$ and $h_2$ both 
cubic. The corresponding situation for quartic polynomials is rather less clear. Consider, for 
example, the irreducible quartic polynomial
$$f_4(t)=t^4+t^2+2t+3.$$
One may computationally confirm that for every non-trivial integral choice of coefficients 
$a,b,c\in \dbZ$ with absolute value at most $1000$, the polynomial $f_4(ax^2+bx+c)$ is 
irreducible, so that no decomposition of the form sought is available in this range. This does 
not rule out the possibility, of course, that there might be a quadratic with very large 
coefficients that does deliver the sought after polysmoothness. On the other hand, if instead 
one works over $\dbQ[t]$ instead of $\dbZ[t]$, then obstructions are possible only for 
biquadratic quartics. Indeed, one has
$$f_4\left( -\frac{x^2+x+3}{2}\right) =\frac{1}{16}(x^4+2x^3+7x^2+2x+9)
(x^4+2x^3+7x^2+10x+13).$$
This example shows that the problem of finding {\it integral} polynomial substitutions 
delivering well-factorability is in general very much more challenging than finding 
corresponding rational polynomial substitutions.\par 

More generally, by completing the fourth power in the usual manner, it is apparent that 
decompositions similar to that of the last paragraph may be obtained for arbitrary quartic 
polynomials provided such is the case for irreducible quartics of the shape 
$f_4(t)=At^4+Bt^2+Ct+D$. If, in addition, one has  $C\ne 0$, then we may put
$$g(x)=-\frac{Ax^2+Bx+D}{C},$$
and we deduce that
$$f_4(g(x))=h_1(x)h_2(x),$$
for suitable quartic irreducible polynomials $h_1,h_2\in \dbQ[x]$. The point here is that, 
if $\alp$ is a root of $f$ in its splitting field, then $\alp=g(\alp^2)$, and so $\alp^2$ is a 
root of $f_4(g(x))$. Thus the minimal polynomial of $\alp^2$ over $\dbQ$ divides 
$f_4(g(x))$. A straightforward exercise confirms that this minimal polynomial is not quadratic, 
whence $[\dbQ(\alp^2):\dbQ]=[\dbQ(\alp):\dbQ]=4$ and the assertion that $h_1$ and 
$h_2$ are quartic follows.\par

For degrees $d$ exceeding $4$, obstructions to these quadratic-based decompositions 
appear, as can be seen from the following criterion. 

\begin{theorem}\label{theorem6.1} Let $f_d\in \dbZ[t]$ be an irreducible polynomial of 
degree $d$ with lead coefficient $A$, and define $\phi_d(x,y)=y^df_d(x/y)$. Suppose that 
there exists a quadratic polynomial $g\in \dbQ[t]$ having the property that $f_d(g(t))$ is 
reducible. Then the equation $Az^2=\phi_d(x,y)$ possesses a solution $(x,y,z)\in \dbQ^3$ 
with $yz\ne 0$.
\end{theorem}

\begin{proof} By a now familiar argument, it is apparent that if $f_d(g(t))=h_1(t)h_2(t)$ is a 
factorisation of $f_d(g(t))$ with $h_i\in \dbQ[t]$ and $\text{\rm deg}(h_i)\ge 1$ $(i=1,2)$, 
then one must have $\text{deg}(h_i)\ge d$. It follows, in particular, that $h_1$ and $h_2$ 
are both constant multiples of irreducible polynomials of degree $d$. Let $\alp$ be a root of 
$f_d$ in its splitting field, and let $\bet$ be a root of the polynomial $g(t)-\alp$ in its splitting 
field. Then since $f(g(\bet))=0$, one must have $h_i(\bet)=0$ for either $i=1$ or $i=2$. 
Also, one has $\dbQ(\alp)=\dbQ(g(\bet))\subseteq \dbQ(\bet)$, and yet 
$[\dbQ(\bet):\dbQ]=\text{\rm deg}(h_i)=d=[\dbQ(\alp):\dbQ]$, so that 
$\dbQ(\bet)=\dbQ(\alp)$.\par

We may write $g(t)=at^2+bt+c$ for some $a,b,c\in \dbQ$ with $a\ne 0$. Thus, we have
$$(2a\bet+b)^2=4ag(\bet)+b^2-4ac=b^2-4ac+4a\alp.$$
Writing $K=\dbQ(\alp)$, and putting $m=4a$ and $n=b^2-4ac$, we obtain the relation
$$\left(\text{\rm Norm}_{K/\dbQ}(2a\bet+b)\right)^2=\text{\rm Norm}_{K/\dbQ}
(m\alp+n)=A^{-1}\phi_d(n,-m).$$
Thus, on recalling that $\bet\in \dbQ(\alp)$, we find that the equation $Az^2=\phi_d(x,y)$ 
has the rational solution
$$z=\text{\rm Norm}_{K/\dbQ}(2a\bet+b), \quad x=n, \quad y=-m\ne 0.$$
This completes the proof of the theorem.
\end{proof}

Note that since $A=\phi_d(1,0)$, the equation $Az^2=\phi_d(x,y)$ has the trivial solution 
$(x,y,z)=(1,0,1)$, and hence is automatically locally soluble everywhere. Of importance for 
the discussion of this section is the connection with hyperelliptic curves. When $d$ is even, 
say $d=2k$, any solution $(x,y,z)\in \dbQ^3$ of this equation with $yz\ne 0$ gives a rational 
point on the hyperelliptic curve defined by the equation $AY^2=\phi_{2k}(X,1)$, namely
$$(X,Y)=\left( \frac{x}{y},\frac{z}{y^k}\right).$$
However, as has been shown by Bhargava (see \cite{Bha2013}, and also \cite{BGW2017} 
for subsequent developments), most hyperelliptic curves over $\dbQ$ have no rational 
points. Thus, we must expect that for most irreducible polynomials $f_d\in \dbQ[x]$ of even 
degree $d\ge 6$, the composition $f_d(g(x))$ should be irreducible for all quadratic 
polynomials $g\in \dbQ[x]$. Specific examples can be obtained with some computational 
effort. For example, one may check that the polynomials
$$F_1(x)=x^6-x^4-21x^2-31\quad \text{and}\quad F_2(x)=x^6+x^4-18x^2-43$$
are irreducible over $\dbQ[x]$. We verified this assertion ourselves by applying the PARI/GP 
software package. Next, by reference to the tables of elliptic curves provided by the 
$L$-functions and Modular Forms Database (available at www.lmfdb.org), one finds that 
the elliptic curves with Weierstrass forms
\begin{equation}\label{6.1}
y^2=x^3-x^2-21x-31\quad \text{and}\quad y^2=x^3+x^2-18x-43,
\end{equation}
with respective Cremona labels 76a1 and 92a2, both have rank $0$ and trivial torsion. These 
elliptic curves consequently have only the single rational point at infinity. In particular, it 
follows that there is no rational solution to either of the equations obtained by substituting 
$(x,y)=(X^2,Y)$ into the equations (\ref{6.1}), namely
$$Y^2=X^6-X^4-21X^2-31\quad \text{and}\quad Y^2=X^6+X^4-18X^2-43.$$
Thus, the above discussion shows that for $i=1$ and $2$, the polynomial $F_i(g(x))$ is 
irreducible for all quadratic polynomials $g\in \dbQ[x]$. One might complain that these two 
examples are rather special, since the Galois group associated with these polynomials is not 
the full symmetric group $S_6$. We are grateful to Michael Stoll for supplying the additional 
example
$$F_3(X)=X^6-3X^5-4X^4+X^3-2X^2-2.$$
This polynomial is ``generic'', in the sense that it is irreducible with Galois group $S_6$, and 
moreover the equation $Y^2=F_3(X)$ has no rational solutions. Thus we may conclude as 
above that for all quadratic polynomials $g\in \dbQ[x]$, the polynomial $F_3(g(x))$ is 
irreducible.

\bibliographystyle{amsbracket}

\begin{thebibliography}{18}

\bibitem{BW1998}
A. Balog and T. D. Wooley, \emph{On strings of consecutive integers with no large prime 
factors}, J. Austral. Math. Soc. Ser. A \textbf{64} (1998), no. 2, 266--276.

\bibitem{Bha2013}
M. Bhargava, \emph{Most hyperelliptic curves over $\dbQ$ have no rational points}, available 
as arXiv:1308.0395.

\bibitem{BGW2017}
M. Bhargava, B. H. Gross and X. Wang, \emph{A positive proportion of locally soluble 
hyperelliptic curves over $\dbQ$ have no point over any odd degree extension} (with an 
appendix by T. Dokchitser and V. Dokchitser), J. Amer. Math. Soc. \textbf{30} (2017), no. 2, 
451--493. 

\bibitem{DMT2001}
C. Dartyge, G. Martin and G. Tenenbaum, \emph{Polynomial values free of large prime 
factors}, Periodica Math. Hungarica \textbf{43} (2001), no. 1-2, 111--119.

\bibitem{GP2006}
A. Granville and P. Pleasants, \emph{Aurifeuillian factorization}, Math. Comp. \textbf{75} 
(2006), no. 253, 497--508.

\bibitem{Har2012}
J. Harrington, \emph{On the factorization of the trinomials $x^n+cx^{n-1}+d$}, Int. J. 
Number Theory \textbf{8} (2012), no. 6, 1513--1518.

\bibitem{Lju1960}
W. Ljunggren, \emph{On the irreducibility of certain trinomials and quadrinomials}, Math. 
Scand. \textbf{8} (1960), 65--70.

\bibitem{Mar2002}
G. Martin, \emph{An asymptotic formula for the number of smooth values of a polynomial}, 
J. Number Theory \textbf{93} (2002), no. 2, 108--182.

\bibitem{Sch1967}
A. Schinzel, \emph{On two theorems of Gelfond and some of their applications}, Acta Arith. 
\textbf{13} (1967), 177--236.

\bibitem{Sch2000}
A. Schinzel, \emph{Polynomials with special regard to reducibility}, in: Encyclopedia of 
mathematics and its applications, vol. 77, Cambridge University Press, Cambridge, 2000.

\bibitem{Sel1956}
E. S. Selmer, \emph{On the irreducibility of certain trinomials}, Math. Scand. \textbf{4} 
(1956), 287--302.

\end{thebibliography}
\providecommand{\bysame}{\leavevmode\hbox to3em{\hrulefill}\thinspace}

\end{document}